\documentclass[reqno]{amsart}
\usepackage{amssymb} 
\usepackage{amscd}
\usepackage{bbm}
\usepackage{color}
\usepackage{esint}
\usepackage{enumerate}
\usepackage{framed}
\usepackage{graphicx}
\usepackage{mathrsfs} 
\usepackage{hyperref} 
\usepackage{amsrefs}

\theoremstyle{plain} 
\newtheorem{theorem}{Theorem}[section]
\newtheorem{proposition}[theorem]{Proposition}
\newtheorem{lemma}[theorem]{Lemma}

\theoremstyle{definition}

\theoremstyle{remark} 
\newtheorem{remark}[theorem]{Remark}

\definecolor{shadecolor}{rgb}{1,0.8,0.3}
\newcommand{\DETAIL}[1]{}
\newcommand{\IGNORE}[1]{}

\newcommand{\C}{{\mathscr{C}}}
\newcommand{\D}{{\mathscr{D}}}

\renewcommand{\H}{{\mathscr{H}}}
\renewcommand{\L}{{\mathscr{L}}}

\newcommand{\N}{{\mathbb{N}}}

\newcommand{\R}{{\mathbb{R}}}

\newcommand{\CB}{{\C_\mathrm{b}}}

\newcommand{\GS}{\geqslant}

\newcommand{\ID}{{\mathrm{id}}}
\newcommand{\LS}{\leqslant}

\newcommand{\SM}{{\mathscr{M}}}
\newcommand{\SP}{{\mathscr{P}}}

\newcommand{\DST}{\displaystyle}

\newcommand{\EPS}{\varepsilon}

\newcommand{\LEB}{{\mathcal{L}}}
\newcommand{\LIP}{{\mathrm{Lip}}}

\newcommand{\RHO}{\varrho}

\newcommand{\TAN}{{\mathbb{T}}}

\newcommand{\COPT}{{\mathcal{C}}}

\newcommand{\PROJ}[1]{\mathsf{P}_{#1}}

\newcommand{\WEAK}{\DOTSB\protect\relbar\protect\joinrel\rightharpoonup}

\DeclareMathOperator{\SPT}{{\mathrm{spt}}}

\numberwithin{equation}{section}

\title[The Pressureless Gas Dynamics Equations]{A Simple Proof of Global
Existence for the 1D Pressureless Gas Dynamics Equations}

\author
  {Fabio Cavalletti}
\address
  {Fabio Cavalletti,
   Lehrstuhl f\"{u}r Mathematik (Analysis),
   RWTH Aachen University,
   Templergraben 55,
   D-52062 Aachen, 
   Germany}
\email
  {cavalletti@instmath.rwth-aachen.de}

\author
  {Marc Sedjro}
\address
  {Marc Sedjro,
   Lehrstuhl f\"{u}r Mathematik (Analysis),
   RWTH Aachen University,
   Templergraben 55,
   D-52062 Aachen, 
   Germany}
\email
  {sedjro@instmath.rwth-aachen.de}

\author
  {Michael Westdickenberg}
\address
  {Michael Westdickenberg,
   Lehrstuhl f\"{u}r Mathematik (Analysis),
   RWTH Aachen University,
   Templergraben 55,
   D-52062 Aachen, 
   Germany}
\email
  {mwest@instmath.rwth-aachen.de}

\date{\today}

\begin{document}

\begin{abstract}
Sticky particle solutions to the one-dimensional pressureless gas
dynamics equations can be constructed by a suitable metric projection
onto the cone of monotone maps, as was shown in recent work by Natile
and Savar\'{e}. Their proof uses a discrete particle approximation and
stability properties for first order differential inclusions. Here we
give a more direct proof that relies on a result by Haraux on the
differentiability of metric projections. We apply the same method also
to the one-dimensional Euler-Poisson system, obtaining a new proof for
the global existence of weak solutions.
\end{abstract}

\maketitle
\tableofcontents


\section{Introduction}\label{S:I}

\newcommand{\MM}{{\mathsf{m}}}

The one-dimensional pressureless gas dynamics equations
\begin{equation}
\label{E:PGDE}
  \left.\begin{aligned}
    \partial_t\RHO + \partial_x(\RHO v) &= 0
\\
    \partial_t(\RHO v) + \partial_x(\RHO v^2) &= 0
  \end{aligned}\right\}
  \quad\text{in $[0,\infty)\times\R$}
\end{equation}
describe the dynamics of a mass (or electric charge) distribution that
moves freely in the absence of any external or internal forces. The
quantity $\RHO$ is a positive Borel measure depending on time/space
$(t,x)$, while $v$ is the Eulerian velocity field. We will assume in
the following that $\RHO(t,\cdot)$ is a probability measure for all
times. This assumption is consistent with first equation in
\eqref{E:PGDE}, called the continuity equation, which describes the
local conservation of mass. The second equation in \eqref{E:PGDE}
models the local conservation of momentum. We are interested in the
Cauchy problem, so we assume that an initial density $\bar{\RHO}$
(which is a Borel probability measure) and an initial Eulerian
velocity $\bar{v} \in \L^2(\R,\bar{\RHO})$ (with finite kinetic
energy) are given.

System \eqref{E:PGDE} is a building block for semiconductor models.
Its multi-dimensional version has been proposed as a simple model
describing the formation of galaxies in the early stage of the
universe. Since there is no pressure, the fluid elements do not
interact with each other. For the applications, however, one typically
augments the system \eqref{E:PGDE} with the assumption of adhesion
(also called sticky particle) dynamics; see \cite{Zeldovich}: Whenever
any fluid elements meet at the same location, they stick together to
form larger compounds. The density $\RHO$ can therefore have singular
parts (Dirac measures). While both mass and momentum are conserved in
the collisions, kinetic energy may be lost. The assumption of adhesion
dynamics can be understood as an entropy condition for the hyperbolic
conservation law \eqref{E:PGDE}.

There is now a huge amount of literature studying the pressureless gas
dynamics equations \eqref{E:PGDE} and establishing global existence of
solutions. Frequently, a sequence of approximate solutions is
constructed by considering discrete particles: The initial mass
distribution is approximated by a finite sum of Dirac measures. The
dynamics of these particles are described by a finite dimensional
system of ordinary differential equations between collision times.
Whenever particles collide, then the new velocity of the bigger
particle is determined from the conservation of mass/momentum and the
assumption of the correct impact law. The general existence result is
obtained by sending the number of discrete particles to infinity. To
pass to the limit, several approaches are feasible. We only mention
two: Brenier and Grenier \cite{BrenierGrenier} consider the cumulative
distribution function $M$ associated to the density $\RHO$:
$$
  M(t,x) := \int_{-\infty}^x \RHO(t,dz)
  \quad\text{for all $(t,x)\in [0,\infty)\times\R$,}
$$
and show that $M$ is the unique entropy solution of a scalar
conservation law
$$
  \partial_t M + \partial_x A(M) = 0
  \quad\text{in $[0,\infty)\times\R$,}
$$
where the flux function $A \colon [0,1]\longrightarrow\R$ depends on
the initial density and Eulerian velocity. In particular, the function
$M$ satisfies Oleinik's entropy condition.

A second approach, introduced by Natile and Savar\'{e}
\cite{NatileSavare}, uses the theory of first-order differential
inclusions on the space of monotone transport maps from some reference
measure space $([0,1],\LEB^1|_{[0,1]}) =: (\Omega,\MM)$ (where
$\LEB^1$ is the one-dimensional Lebesgue measure) to $\R$. To every
density/velocity $(\RHO,v)$ that solves \eqref{E:PGDE} one can
associate a unique map $X \in \L^2(\Omega,\MM)$ with $X$ {\em
monotone} such that
\begin{equation}
\label{E:COMP}
  \RHO(t,\cdot) = X(t,\cdot) \# \MM
  \quad\text{for all $t\in[0,\infty)$.}
\end{equation}
Here $\#$ indicates the push-forward of measures. Then $X$ satisfies
\begin{equation}
\label{E:FODI}
  \dot{X} + \partial I_\COPT(X) \ni \bar{V}
  \quad\text{for all $t\in[0,\infty)$,}
\end{equation}
where $\COPT$ denotes the closed convex cone of all transport maps $X
\in \L^2(\Omega,\MM)$ that are monotone, and where $\partial I_\COPT$
is the subdifferential of the indicator function of $\COPT$. If $X$
satisfies \eqref{E:FODI} and is related to $\RHO$ through
\eqref{E:COMP}, then the Eulerian velocity $v$ can be recovered from
the Lagrangian velocity $V := \dot{X}$ through
\begin{equation}
\label{E:VELOS}
  V(t,\cdot) = v(t,X(t,\cdot))
  \quad\text{for all $t\in[0,\infty)$.}
\end{equation}
Assuming finite kinetic energy, it is natural to require that
$$
  V(t,\cdot) \in \L^2(\Omega,\MM),
  \quad
  v(t,\cdot) \in \L^2(\R,\RHO(t,\cdot)).
$$
The relation \eqref{E:VELOS} in particular determines the initial
Lagrangian velocity $\bar{V}$ in \eqref{E:FODI} in terms of the
initial data $(\RHO,v) (0,\cdot) =: (\bar{\RHO},\bar{v})$ of the
system \eqref{E:PGDE}.

Integrating \eqref{E:FODI} in time and using the fact that the
assumption of sticky particles implies that the subdifferential
$\partial I_\COPT(X(t,\cdot))$ is nondecreasing in time (as a subset
of $\L^2(\Omega,\MM)$), one can show that the solution of
\eqref{E:FODI} can be written as
\begin{equation}
\label{E:PROJ}
  X(t,\cdot) = P_\COPT( \bar{X}+t\bar{V} )
  \quad\text{for all $t\in[0,\infty)$,}
\end{equation}
with $\bar{X} := X(0,\cdot)$ specified by \eqref{E:COMP}; see
\cite{NatileSavare}. Here $P_\COPT$ denotes the metric projection onto
the cone $\COPT$. The connection between \eqref{E:PGDE} and
\eqref{E:FODI} makes it possible to apply classical results from the
theory of first-order differential inclusions in Hilbert spaces to
study the pressureless gas dynamics equations; see
\cites{NatileSavare, Brezis} for more.

In particular, one can use stability results for first-order
differential inclusions to prove that the discrete particle
approximation outlined above generates (a certain class of) solutions
for \eqref{E:PGDE}, as the number of particles converges to infinity.
Similar results can be obtained by applying the scalar conservation
law approach by Brenier and Grenier. We also refer the reader to
\cites{BouchutJames1, Grenier, ERykovSinai, PoupaudRascle,
BouchutJames2, HuangWang, Wolansky, Moutsinga, NguyenTudorascu,
GangboNguyenTudorascu, BrenierGangboSavareWestdickenberg}.

In this paper, we will show directly (that is, without passing through
a discrete particle approximation first) that the representation
\eqref{E:PROJ} generates solutions of the pressureless gas dynamics
equations. To this end, we need a good understanding of the derivative
of the map $t\mapsto X(t,\cdot)$, thus of the derivative of the metric
projection away from the boundary of $\COPT$. A classical result by
Haraux \cite{Haraux} gives a variational characterization of
$\dot{X}$, which in our case implies that $V$ is exactly of the form
\eqref{E:VELOS}, where $v$ is a suitable Eulerian velocity. Our
approach is therefore a bit simpler than the ones mentioned above,
which rely either on the theory of entropy solutions of scalar
conservation laws, or alternatively on the theory of first-order
differential inclusions. We will show how our method can be
generalized to the case when \eqref{E:PGDE} contains force terms. In
particular, we consider the Euler-Poisson system.


\section{Differentiability of Metric Projections}\label{S:DOMP}

Let $\COPT$ be a closed convex subset in some Hilbert space $\H$,
equipped with scalar product $\langle\cdot, \cdot\rangle$ and induced
norm $\|\cdot\|$. For any $Y \in \H$ we will denote by
$\PROJ{\COPT}(Y)$ the metric projection of $Y$ onto the cone $\COPT$,
so that $\PROJ{\COPT}(Y)$ satisfies
$$
  \|Y-\PROJ{\COPT}(Y)\| = \inf\Big\{ \|Y-Z\| \colon Z\in\COPT \Big\}.
$$
It is well-known that the metric projection $\PROJ{\COPT}(Y)$ exists
and is uniquely determined for all $Y\in\H$. Moreover, it is
characterized by the following property:
\begin{equation}
\label{E:CHAR}
  Y^* = \PROJ{\COPT}(Y)
  \quad\Longleftrightarrow\quad
  \begin{cases}
    Y^* \in \COPT,
\\
    \langle Y-Y^*,Y^*-Z \rangle \GS 0 & \text{for all $Z\in\COPT$.}
  \end{cases}
\end{equation}
If $\COPT$ is a cone, then one can choose $Z=0$ and $Z=2Y^*$ in
\eqref{E:CHAR} to obtain
\begin{equation}
\label{E:PROJ2}
  \langle Y-Y^*,Y^* \rangle = 0,
  \quad
  \langle Y-Y^*,Z \rangle \LS 0
  \quad\text{for all $Z\in\COPT$.}
\end{equation}
It is also well-known that the metric projection is a contraction:
\begin{equation}
\label{E:CONTR}
  \|\PROJ{\COPT}(Y_1)-\PROJ{\COPT}(Y_2)\| \LS \|Y_1-Y_2\|
  \quad\text{for all $Y_1,Y_2\in\H$.}
\end{equation}
We refer the reader to \cite{Zarantonello} for further information on
metric projections.

For given $\bar{X}, \bar{V} \in \H$ we now consider the map $t \mapsto
X_t := \PROJ{\COPT}(\bar{X}+t\bar{V})$ for $t\in\R$, which is
well-defined and Lipschitz continuous because of \eqref{E:CONTR}. We
have
\begin{equation}
\label{E:LIPS}
  \|X_{t+h}-X_t\| \LS |h| \|\bar{V}\|
  \quad\text{for all $t,h\in\R$.}
\end{equation}
The velocity $V_t := \lim_{h\rightarrow 0} (X_{t+h}-X_t)/h$ exists
strongly for a.e.\ $t\in\R$ and satisfies the inequality $\|V_t\| \LS
\|\bar{V}\|$. For any $Z\in\COPT$ we define the tangent cone
\begin{equation}
\label{E:TN}
  \TAN_Z\COPT := \overline{T_Z\COPT},
  \quad
  T_Z\COPT := \bigcup_{h>0} h(\COPT-Z).
\end{equation}
Note that $T_Z\COPT$ is a convex set: if $h_1,h_2>0$ and
$Y_1,Y_2\in\COPT$ are given, then
$$
  (1-\lambda) h_1(Y_1-Z) + \lambda h_2(Y_2-Z)
    = h \Big( \big( (1-\mu) Y_1 + \mu Y_2 \big) - Z \Big)
    \in T_Z\COPT
$$
for all $\lambda\in[0,1]$, where $h := (1-\lambda)h_1+\lambda h_2 > 0$
and $\mu := \lambda h_2/h \in [0,1]$. For any $Y\in\H$ we denote by
$[Y]^\perp$ the orthogonal complement of $\R Y$. 
\medskip

We will rely on the following result:

\begin{theorem}
\label{T:HARAUX}
For fixed $t\in\R$, let $V_t$ be any weak limit point of
$$
  V_t(h_n) := (X_{t+h_n}-X_t)/h_n
$$
as $h_n\rightarrow 0+$. Then $V_t \in \Sigma_{X_t}\COPT$ and $\langle
\bar{V}-V_t, V_t \rangle \GS 0$, where
\begin{equation}
\label{E:SIGMA}
  \Sigma_{X_t}\COPT := \TAN_{X_t}\COPT \cap [(\bar{X}+t\bar{V})-X_t]^\perp.
\end{equation}
Moreover, we have $\langle \bar{V}-V_t, W\rangle \LS 0$ for all $W \in
T_{X_t}\COPT \cap [(\bar{X}+t\bar{V})-X_t]^\perp$.
\end{theorem}

\begin{remark}
Theorem~\ref{T:HARAUX} follows from Proposition~1 in \cite{Haraux}. We
include the short proof below for the reader's convenience. If the
cone $\COPT$ is {\em polyhedric}, so that
$$
  \overline{T_{X_t}\COPT \cap [(\bar{X}+t\bar{V})-X_t]^\perp} 
    = \Sigma_{X_t}\COPT
$$
for all $t\in\R$, then the map $t\mapsto X_t$ is strongly
right-differentiable and, denoting the right derivative again by
$V_t$, we have $V_t = \PROJ{\Sigma_{X_t}\COPT}(\bar{V})$ for all
$t\in\R$. Recall that the metric projection onto a closed convex set
is uniquely determined, so the weak limit points of $V_t(h_n)$ for all
sequences $h_n\rightarrow 0+$ have the same limit. In order to prove
the claim it is sufficient to observe that, by continuity, we have
$\langle \bar{V}-V_t, W \rangle \LS 0$ for all
$W\in\Sigma_{X_t}\COPT$. In particular, this implies that $\langle
\bar{V}-V_t, V_t \rangle = 0$.
\end{remark}

\begin{proof}[Proof of Theorem~\ref{T:HARAUX}]
The difference quotients $V_t(h)$ are uniformly bounded in $\H$ for
all $t,h\in\R$ (because of \eqref{E:LIPS}), and therefore weakly
precompact. Let $h_n \rightarrow 0+$ be any sequence such that the
weak limit $V_t := \lim_{n\rightarrow\infty} V_t(h_n)$ exists. Then
\begin{align*}
  0 & \GS \Big\langle \big( \bar{X}+(t+h_n)\bar{V} \big)-X_{t+h_n}, 
      X_t-X_{t+h_n} \Big\rangle
\\
    & = \Big\langle \big( (\bar{X}+t\bar{V})+h_n\bar{V} \big) 
      -\big( h_n V_t(h_n)+X_t \big), X_t-\big(h_n V_t(h_n)+X_t) \Big\rangle
\\
    & = -h_n \Big\langle (\bar{X}+t\bar{V})-X_t, V_t(h_n) \Big\rangle
        -h_n^2 \Big\langle \bar{V}-V_t(h_n), V_t(h_n) \Big\rangle;
\end{align*}
see \eqref{E:CHAR}. We used that $X_{t+h_n} = X_t+h_n V_t(h_n)$ and
$X_t\in\COPT$. This implies that
\begin{align}
  -h_n^2 \Big\langle \bar{V}-V_t(h_n),V_t(h_n) \Big\rangle
    & \LS h_n \Big\langle (\bar{X}+t\bar{V})-X_t, V_t(h_n) \Big\rangle
\label{E:TRI}\\ 
    & = \Big\langle (\bar{X}+t\bar{V})-X_t, X_{t+h_n}-X_t \Big\rangle \LS 0.
\nonumber
\end{align}
We now divide by $h_n^2>0$ and let $n\rightarrow\infty$ to obtain
$-\langle \bar{V}-V_t, V_t\rangle \LS 0$, using
$$
  \langle \bar{V},V_t(h_n) \rangle \longrightarrow \langle \bar{V}, V_t \rangle
  \quad\text{and}\quad
  \|V_t\|^2 \LS \liminf_{n\rightarrow\infty} \|V_t(h_n)\|^2.
$$
Dividing \eqref{E:TRI} by $h_n>0$ instead and letting
$n\rightarrow\infty$, we obtain
\begin{equation}
\label{E:ORTH}
  \langle (\bar{X}+t\bar{V})-X_t, V_t\rangle = 0
\end{equation}
because the $V_t(h_n)$ are uniformly bounded. On the other hand, since
$T_{X_t}\COPT$ (recall \eqref{E:TN}) is convex, its weak and strong
closures coincide. Then the weak limit
$$
  V_t = \lim_{n\rightarrow\infty}(X_{t+h_n}-X_t)/h_n
$$
satisfies $V_t \in \TAN_{X_t}\COPT$. It follows that $V_t \in
\Sigma_{X_t}\COPT$.

Consider now any $Z\in\COPT$ with the property that
\begin{equation}
\label{E:ORTH2}
  \langle (\bar{X}+t\bar{V})-X_t,Z-X_t \rangle = 0.
\end{equation}
Let $\delta_n := V_t(h_n)-V_t$ so that $\delta_n\WEAK 0$ as
$n\rightarrow\infty$. Then
\begin{align*}
  0 & \GS \Big\langle \big( \bar{X}+(t+h_n)\bar{V} \big)-X_{t+h_n},
      Z-X_{t+h_n} \Big\rangle
\\
    & = \Big\langle \big( (\bar{X}+t\bar{V})-X_t \big)
      + h_n(\bar{V}-V_t) - h_n\delta_n,
      (Z-X_t) - h_n V_t - h_n\delta_n \Big\rangle.
\end{align*}
Rearranging terms and dividing by $h_n>0$, we obtain
\begin{align*}
  & \langle \bar{V}-V_t, Z-X_t \rangle
\\
  & \qquad
    \LS -\big\langle (\bar{X}+t\bar{V})-X_t, Z-X_t \big\rangle / h_n
\\
  & \qquad\quad
    + \Big( \big\langle (\bar{X}+t\bar{V})-X_t, V_t \big\rangle
        + \big\langle (\bar{X}+t\bar{V})-X_t,\delta_n \big\rangle
        + \langle \delta_n, Z-X_t \rangle \Big)
\\
  & \qquad\quad
    + h_n \Big( \langle \bar{V}-V_t, V_t \rangle
        + \langle \bar{V}-V_t,\delta_n \rangle
        - \langle \delta_n,V_t \rangle
        - \|\delta_n\|^2 \Big).
\end{align*}
Passing to the limit on the right-hand side and using
\eqref{E:ORTH2}/\eqref{E:ORTH}, we obtain
$$
  \langle \bar{V}-V_t, Z-X_t \rangle \LS 0
  \quad\text{for all $Z\in\COPT$.}
$$
We observe now that for any $W \in T_{X_t}\COPT \cap
[(\bar{X}+t\bar{V})-X_t]^\perp$ there exist $h>0$ and $Z\in\COPT$ with
$W = h(Z-X_t)$ and \eqref{E:ORTH2}. Then $\langle \bar{V}-V_t, W
\rangle \LS 0$ follows.
\end{proof}


\section{Sticky Particle Solutions}\label{S:SPS}

We now apply the result of the previous section to the following
situation: the Hilbert space $\H = \L^2(\R,\mu)$ for some reference
measure $\mu\in\SP_2(\R)$, where $\SP_2(\R)$ denotes the space of all
Borel probability measures with finite second moment (thus $\int_\R
|x|^2 \,\mu(dx) < \infty$). The inner product
$\langle\cdot,\cdot\rangle$ and the norm $\|\cdot\|$ are the
$\L^2(\R,\mu)$-inner product and -norm, respectively. The cone is
defined as
\begin{equation}
\label{E:COPT}
  \COPT := \big\{ X\in\L^2(\R,\mu) \colon \text{$X$ is monotone} \big\}.
\end{equation}
We say that $X$ is monotone if the support of the induced transport
plan
\begin{equation}
\label{E:GAMMAX}
  \gamma_X := (\ID,X)\#\mu
\end{equation}
is a monotone set in $\R\times\R$, where $\#$ denotes the push-forward
of measures. This definition is motivated by the theory of optimal
transport, where a transport plan (a probability measure of the
product space) is optimal if its support is contained in the
subdifferential of a convex function (see \cite{AmbrosioGigliSavare}),
which is monotone. Recall that a subset $\Gamma \subset \R\times\R$ is
monotone if for any $(m_i,x_i) \in \Gamma$, $i=1..2$, we have
$$
  (m_1-m_2)(x_1-x_2) \GS 0.
$$
For a Borel measure $\nu$ defined on some topological space $\Omega$,
we say $z\in\SPT\nu$ if
$$
  \nu(N)>0 
  \quad\text{for every open neighborhood $N$ of $z$.}
$$
Consequently, a point $(m,x)\in\R^d\times\R^d$ belongs to the support
$\SPT\gamma_X$ if
\begin{equation}
\label{E:SPA}
  \gamma_X\big( B_\delta(m)\times B_\delta(x) \big) > 0
  \quad\text{for all $\delta>0$.}
\end{equation}
%

\begin{lemma}\label{L:SUPPORT}
For given reference measure $\mu\in\SP_2(\R)$ and $X\in\COPT$ we
define $\gamma_X$ as in \eqref{E:GAMMAX}. Then there exists a Borel
set $N_X \subset \R$ with $\mu(N_X) = 0$ such that
\begin{equation}
\label{E:SPP}
  (m,X(m)) \in \SPT\gamma_X
  \quad\text{for all $m\in\R\setminus N_X$.}
\end{equation}
\end{lemma}

\begin{proof}
It is known that every finite Borel measure $\nu$ on a locally compact
Hausdorff space $\Omega$ with countable basis is inner regular; see
e.g.\ \cite{Folland} for more details. It follows that $\nu(U) = 0$,
where $U := \Omega\setminus \SPT\nu$. Indeed we have
\begin{equation}
\label{E:EINS}
  \nu(U) = \sup\Big\{ \nu(K) \colon \text{$K\subset U$ compact} \Big\},
\end{equation}
by inner regularity of $\nu$. For any $x\in U$ there exists an open
neighborhood $N_x$ of $x$ with $\nu(N_x) = 0$, as follows from the
definition of $\SPT\nu$. The family $\{ N_x \}_{x\in U}$ is an open
covering of the compact set $K$, and so a finite subcovering exists:
Let $V\subset U$ be a finite set such that $K \subset \bigcup_{x\in V}
N_x$. Then we can estimate
\begin{equation}
\label{E:ZWEI}
  \nu(K) \LS \sum_{x \in V} \nu(N_x) = 0.
\end{equation}
Combining this inequality with \eqref{E:EINS}, we obtain that
$\nu(U)=0$.

Applying this observation to transport plans, we obtain \eqref{E:SPP}.
Just note that
\begin{align*}
  \mu\Big( \big\{ m\in\R \colon (m,X(m))\not\in\SPT\gamma_X \big\} \Big)
    &= \mu\Big( (\ID\times X)^{-1}
      \big( (\R\times\R)\setminus\SPT\gamma_X \big) \Big)
\\
    &= \gamma_X\Big( (\R\times\R)\setminus\SPT\gamma_X \Big) = 0,
\end{align*}
by the definition of push-forward of measures.
\end{proof}

\begin{lemma}\label{L:SMOOTH}
The set $\COPT$ defined in \eqref{E:COPT} is a closed convex cone in
$\L^2(\R,\mu)$. For any $X\in\COPT$ and any smooth, strictly
increasing function $\zeta \colon \R\longrightarrow\R$ that coincides
with the identity map outside a compact set, we have $\zeta\circ X \in
\COPT$.
\end{lemma}

\begin{proof}
We proceed in two steps.
\medskip

{\bf Step~1.} Consider a sequence of $X^k\in\COPT$ converging strongly
to $X \in \L^2(\R,\mu)$. The sequence of transport plans
$\gamma_{X^k}$, as defined in \eqref{E:GAMMAX}, converges narrowly to
$\gamma_X$; see Lemma~5.4.1 in \cite{AmbrosioGigliSavare}. Narrow
convergence implies convergence in the Kuratowski sense (see
Proposition~5.1.8 in \cite{AmbrosioGigliSavare}): for any pair of
points $(\bar{m}_i,\bar{x}_i) \in \SPT\gamma_X$, $i=1..2$, there exist
sequences of $(m_i^k,x_i^k) \in \SPT\gamma_{X^k}$ such that
$$
  (m_i^k,x_i^k) \longrightarrow (\bar{m}_i,\bar{x}_i)
  \quad\text{as $k\rightarrow\infty$,}
$$
for $i=1..2$. Since $\SPT_{X^k}$ is monotone, by definition of
$\COPT$, we have
$$
  (\bar{m}_1-\bar{m}_2)(\bar{x}_1-\bar{x}_2) 
    = \lim_{k\rightarrow\infty} (m_1^k-m_2^k)(x_1^k-x_2^k)
    \GS 0.
$$
This implies that also $\SPT\gamma_X$ is monotone and $X\in\COPT$.
Therefore $\COPT$ is closed.

Let now $X^1, X^2 \in \COPT$ and $s\in[0,1]$ be given and define $X_s
:= (1-s)X^1 + sX^2$, which is an element in $\L^2(\R,\mu)$. Let
$\gamma_{X_s}$ be corresponding transport plan defined by
\eqref{E:GAMMAX}, and let $N^k$ be the $\mu$-null sets in
Lemma~\ref{L:SUPPORT} corresponding to $X^k$, $k=1..2$. Pick any pair
of points $(\bar{m}_i,\bar{x}_i) \in \SPT\gamma_{X_s}$, $i=1..2$. For
any $\delta>0$ there exist
$$
  m_i \in B_\delta(\bar{m}_i)\setminus(N^1\cup N^2)
$$
such that $X_s(m_i) \in B_\delta(\bar{x}_i)$. Indeed assume that not.
Then
\begin{align*}
  0 &= \mu\Big( \big\{ m\in B_\delta(\bar{m}_i)\setminus(N^1\cup N^2) \colon
     X_s(m) \in B_\delta(\bar{x}_i) \big\} \Big)
\\
  & = \mu\Big( \big\{ m\in B_\delta(\bar{m}_i) \colon
     X_s(m) \in B_\delta(\bar{x}_i) \big\} \Big)
\\
  & \vphantom {\Big(}
    = \gamma_{X_s}\big( B_\delta(\bar{m}_i) \times B_\delta(\bar{x}_i) \big),
\end{align*}
which would be a contradiction to our choice of $(\bar{m}_i,\bar{x}_i)
\in \SPT\gamma_{X_s}$. We get that
\begin{align}
  (\bar{m}_1-\bar{m}_2) (\bar{x}_1-\bar{x}_2)
    & \GS (m_1-m_2)\big( X_s(m_1)-X_s(m_2) \big)
\label{E:COM}\\
    & -2\delta\big( |\bar{m}_1|+|\bar{m}_2|+|\bar{x}_1|+|\bar{x}_2| \big)
      -4\delta^2.
\nonumber
\end{align}
But since $(m_i,X^k(m_i)) \in \SPT\gamma_{X^i}$ and $X^k\in\COPT$, and
by definition of $X_s$, it follows that the first term on the
right-hand side of \eqref{E:COM} is nonnegative. Recall that
$\delta>0$ was arbitrary. Therefore the left-hand side of
\eqref{E:COM} is nonnegative as well. Since this construction works
for any $(\bar{m}_i,\bar{x}_i) \in \SPT\gamma_{X_s}$, $i=1..2$, we
have that $\SPT\gamma_{X_s}$ is a monotone set, and thus $X_s \in
\COPT$ for any $s\in[0,1]$. Hence $\COPT$ is convex.

The proof that $\COPT$ is a cone is similar and omitted. 
\medskip

{\bf Step~2.} Consider now a smooth function $\zeta$ as above and
$X\in\COPT$. Choose $M>0$ such that $\zeta(x) = x$ for all $|x|\GS M$
and let $Y := \zeta\circ X$. Then
\begin{align*}
  \int_\R |Y|^2 \,\mu 
    &= \int_{\{|X|<M\}} |Y|^2 \,\mu + \int_{\{|X|\GS M\}} |Y|^2 \,\mu
\\
    &\LS \|\zeta\|^2_{\L^\infty([-M,M])} + \int_\R |X|^2 \,\mu,
\end{align*}
which is finite, and so $Y \in \L^2(\R,\mu)$. With $\gamma_X$ and
$\gamma_Y$ defined as in \eqref{E:GAMMAX}, we must show that
$\SPT\gamma_Y$ is a monotone subset of $\R\times\R$. We first claim
that
$$
  \SPT\gamma_Y = F(\SPT\gamma_X),
  \quad \text{where $F(m,x) := (m,\zeta(x))$}
$$
for all $(m,x)\in\R\times\R$. Notice that $\zeta$ is invertible.
Consider any $(\bar{m}, \bar{y}) \in F(\SPT\gamma_X)$ and let $\bar{x}
:= \zeta^{-1} (\bar{y})$, so that $(\bar{m},\bar{x}) \in
\SPT\gamma_X$. For $\delta>0$ arbitrary we have
$$
  \gamma_Y\big( B_\delta(\bar{m})\times B_\delta(\bar{y})\big)
    = \mu\Big( \big\{ m\in B_\delta(\bar{m}) \colon 
      Y(m) \in B_\delta(\bar{y}) \big\} \Big),
$$
by definition of $\gamma_Y$. Since $\zeta$ is smooth, there exists
$\EPS>0$ with $B_\EPS(\bar{x}) \subset \zeta^{-1}(B_\delta(\bar{y}))$.
Without loss of generality, we may assume that $\EPS\LS\delta$. Then
\begin{align*}
  \gamma_Y\big( B_\delta(\bar{m})\times B_\delta(\bar{y})\big)
     &\GS \mu\Big( \big\{ m\in B_\EPS(\bar{m}) \colon 
      X(m) \in B_\EPS(\bar{x}) \big\} \Big)
\\
     &= \gamma_X\big( B_\EPS(\bar{m})\times B_\EPS(\bar{x})\big),
\end{align*}
which is positive since $(\bar{m},\bar{x}) \in \SPT\gamma_X$; see
\eqref{E:SPA}. It follows that $F(\SPT\gamma_X) \subset \SPT\gamma_Y$.
For the converse direction, we argue analogously, noting that $F$ is
invertible with $F^{-1}(m,y) := (m,\zeta^{-1}(y))$ for all $(m,y) \in
\R\times\R$. Consider now $(m_i,y_i) \in \SPT\gamma_Y$, $i=1..2$. Then
$(m_i,x_i) \in \SPT\gamma_X$ with $x_i:=\zeta^{-1}(y_i)$. Assume
$x_1\neq x_2$. Then
$$
  (y_1-y_2)(m_1-m_2)
    = \frac{\zeta(x_1)-\zeta(x_2)}{x_1-x_2} \cdot (x_1-x_2)(m_1-m_2) 
    \GS 0.
$$
Indeed the first factor is positive since $\zeta$ is strictly
increasing, and the second factor is nonnegative since $\SPT\gamma_X$
is a monotone set. We conclude that $Y\in\COPT$.
\end{proof}

We are going to prove the following result; see \cite{NatileSavare}.

\begin{theorem}[Global Existence]\label{T:PGD}
Let initial data $\bar{\RHO}\in\SP_2(\R)$ and $\bar{v} \in
\L^2(\R,\bar{\RHO})$ of the pressureless gas dynamics equations
\eqref{E:PGDE} be given. For some reference measure $\mu \in
\SP_2(\R)$, let $\bar{X} \in \COPT$ be the unique monotone transport
with $\bar{X}\#\mu = \bar{\RHO}$. Let
\begin{equation}
\label{E:PRF}
  \bar{V} := \bar{v}\circ\bar{X},
  \quad
  X_t := \PROJ{\COPT}(\bar{X}+t\bar{V})
  \quad\text{for all $t\in\R$.}
\end{equation}
Then $X_t$ is differentiable for a.e.\ $t\in\R$ and $V_t := \dot{X}_t$
can be written in the following form: there exists a velocity $v_t \in
\L^2(\R,\RHO_t)$ with $\RHO_t := X_t\#\mu$, such that $V_t = v_t \circ
X_t$. The pair $(\RHO_t, v_t)$ is a weak solution of the conservation
law \eqref{E:PGDE}.
\end{theorem}

\begin{remark}
A comment is in order on the role played by the reference measure
$\mu$. In Theorem~\ref{T:PGD} we assume the existence of a monotone
transport map $\bar{X} \in \COPT$ such that $\bar{X}\#\mu =
\bar{\RHO}$. Such a map may not exist (e.g.\ if the initial density
$\bar{\RHO}$ is a diffuse measure and the reference measure $\mu$ is
concentrated). In the following discussion, no particular properties
of $\mu$ are used. The theorem should hence be understood in the sense
that for any given $\bar{\RHO}$ there exist some reference measure
$\mu$ and a monotone transport map $\bar{X} \in \COPT$ with the
desired property $\bar{X}\#\mu = \bar{\RHO}$.
\end{remark}

Let us start with three lemmas.

\begin{lemma}\label{L:CHR}
For any $X_t$ as in Theorem~\ref{T:PGD} we define 
\begin{equation}
\label{E:HT}
  \H_{X_t} := \text{$\L^2(\R,\mu)$-closure of $\big\{ \varphi\circ X_t \colon
    \varphi\in\D(\R) \big\}$}
\end{equation}
and $\RHO_t := X_t\#\mu$. Then the following statement is true: the
function $W\in\L^2(\R,\mu)$ is in $\H_{X_t}$ if and only if there
exists $w\in\L^2(\R, \RHO_t)$ such that $W=w\circ X_t$.
\end{lemma}

\begin{proof}
Let $W\in\H_{X_t}$ be given and consider a sequence of
$\varphi^k\in\D(\R)$ with
\begin{equation}\label{E:UNO}
  \varphi^k\circ X_t \longrightarrow W
  \quad\text{in $\L^2(\R,\mu)$,}
\end{equation}
so that the $\varphi^k\circ X_t$ form a Cauchy sequence in
$\L^2(\R,\mu)$. Since
$$
  \|\varphi^k\circ X_t\|_{\L^2(\R,\mu)} = \|\varphi^k\|_{\L^2(\R,\RHO_t)},
$$
the $\varphi^k$ then form a Cauchy sequence in $\L^2(\R,\RHO_t)$.
Hence there exists $w\in\L^2(\R,\RHO_t)$ with the property that
$\varphi^k \longrightarrow w$ in $\L^2(\R,\RHO_t)$, by completeness.
This implies
\begin{equation}\label{E:DUE}
  \varphi^k\circ X_t \longrightarrow w\circ X_t
  \quad\text{in $\L^2(\R,\mu)$.}
\end{equation}
Combining \eqref{E:UNO} and \eqref{E:DUE}, we obtain that $W=w\circ
X_t$.

For the converse direction, we recall that any finite Borel measure
$\nu$ on a locally compact Hausdorff space $\Omega$ with continuous
base is inner regular and so the space of all continuous functions
with compact support is dense in $\L^2(\Omega,\nu)$. We refer the
reader to \cite{Folland} for more information. If $\Omega$ is also a
vector space, the same statement is true for smooth functions with
compact support. Let $W=w\circ X_t$ be given and choose a sequence of
$\varphi^k \in \D(\R)$ such that $\varphi^k\longrightarrow w$ in
$\L^2(\R, \RHO_t)$. Since
$$
  \|W-\varphi^k\circ X_t\|_{\L^2(\R,\mu)}
    = \|w-\varphi^k\|_{\L^2(\R,\RHO_t)}
    \longrightarrow 0
$$
for $k\rightarrow\infty$, we conclude that $W\in\H_{X_t}$. This
finishes the proof.
\end{proof}

\begin{lemma}\label{L:SUBSET}
Let $\H_{X_t}$ be defined as in Lemma~\ref{L:CHR}. Then $\H_{X_t}
\subset \mathbb{S}_{X_t}\COPT$, with
$$
  \mathbb{S}_{X_t}\COPT := \text{$\L^2(\R,\mu)$-closure of $T_{X_t}\COPT 
    \cap [(\bar{X}+t\bar{V})-X_t]^\perp$.}
$$
\end{lemma}

\begin{proof}
By density, it is enough to show that $\varphi\circ X_t \in
T_{X_t}\COPT \cap [(\bar{X} +t\bar{V}) -X_t]^\perp$ for all
$\varphi\in\D(\R)$. We choose a constant $h >
\|\varphi'\|_{\L^\infty(\R)}$ and define
$$
  Z^\pm_h := \bigg( \ID\pm\frac{1}{h}\varphi \bigg) \circ X_t,
$$
which is a composition of a (smooth) monotone map with $X_t$, and thus
an element of the cone $\COPT$; see Lemma~\ref{L:SMOOTH}. Rearranging
terms, we obtain the inclusion
$$
  \varphi\circ X_t = h (Z^+_h-X_t) \in T_{X_t}\COPT.
$$
On the other hand, using the characterization \eqref{E:CHAR}, we find
that
$$
  \pm\big\langle (\bar{X}+t\bar{V})-X_t, \varphi\circ X_t \big\rangle
    = h \big\langle (\bar{X}+t\bar{V})-X_t, Z^\pm_h-X_t \big\rangle
    \LS 0,
$$
from which the orthogonality $\varphi\circ X_t \in [(\bar{X}+t\bar{V})
-X_t]^\perp$ follows.
\end{proof}

\begin{lemma}\label{L:CONTAIN}
With the notation above, we have $V_t \in \H_{X_t}$ for a.e.\
$t\in\R$.
\end{lemma}

\begin{proof}  Since the map $t\mapsto X_t$ is Lipschitz continuous, it is
(strongly) differentiable for a.e.\ $t\in\R$. Fix any such $t\in\R$.
Then the velocity $V_t := \dot{X}_t$ satisfies
$$
  V_t = \lim_{h\rightarrow 0+} \frac{X_{t+h}-X_t}{h}
    = -\lim_{h\rightarrow 0+} \frac{X_{t-h}-X_t}{h},
$$
which implies that $V_t \in \TAN_{X_t}\COPT \cap (-\TAN_{X_t}\COPT)$.
We now proceed in three steps. 
\medskip

{\bf Step~1.} Since $X_t\in\COPT$, the support of $\gamma_{X_t} :=
(\ID,X_t)\#\mu$ is a monotone subset of $\R\times\R$, which can be
extended to a maximal monotone set $\Gamma$; see
\cite{AlbertiAmbrosio}. We denote by $u$ the corresponding maximal
monotone set-valued map:
$$
  u(m) := \{ x\in\R \colon (m,x)\in\Gamma \}
  \quad\text{for $m\in\R$.}
$$
Then the level sets of $u$ are all closed intervals; see
\cite{AlbertiAmbrosio}. Consequently, there can be at most countably
many level sets of $u$ that contain more than one point.

Let now $N_{X_t}$ be the $\mu$-null set from Lemma~\ref{L:SUPPORT}
with the property that
$$
  (m,X_t(m)) \in \SPT\gamma_{X_t}
  \quad\text{for all $m\in\R\setminus N_{X_t}$.}
$$
For any $x\in\R$ we define the level set 
$$
  L_x := \Big\{ m\in\R\setminus N_{X_t} \colon X_t(m)=x \Big\}.
$$
Note that the $L_x$ are pairwise disjoint because $X_t$ is a function
(thus single-valued). Since every $L_x$ is contained in the
corresponding $x$-level set of $u$, the set
$$
  \mathscr{O} := \Big\{ x\in\R \colon 
    \text{$L_x$ has more than one element} \Big\}
$$
is at most countable. Then the map $X_t$ is invertible on $\R\setminus
\bigcup_{x\in\mathscr{O}} L_x$. 
\medskip

{\bf Step~2.} Consider any $W\in\TAN_{X_t}\COPT$. Then there exist
$Y^k\in\COPT$, $\lambda^k>0$ with
\begin{equation}
\label{E:STRO}
  W^k := Y^k-\lambda^k X_t \longrightarrow W
  \quad\text{in $\L^2(\R,\mu)$,}
\end{equation}
where we used that $\COPT$ is a cone. Extracting a subsequence if
necessary, we may assume that $W^k \longrightarrow W$ $\mu$-a.e. For
any $k\in\N$ we denote by $N^k$ the $\mu$-null set from
Lemma~\ref{L:SUPPORT}, corresponding to $Y^k \in \COPT$. Then there
exists a $B\subset\R$ with $\mu(B)=0$ such that $W^k(m)
\longrightarrow W(m)$ for every $m\in\R\setminus B$, and $B$ contains
all $N^k$. For any $x\in\mathscr{O}$ consider now $m_1,m_2\in
L_x\setminus B$ with $m_1\LS m_2$. For all $k\in\N$ we have
\begin{align}
  W^k(m_2)-W^k(m_1)
    &= \Big( \big( Y^k(m_2)-\lambda^k X_t(m_2) \big) 
      - \big( Y^k(m_1)-\lambda^k X_t(m_1) \big) \Big)
\nonumber\\
    &= Y^k(m_2)-Y^k(m_1)
\label{E:YS}
\end{align}
because $X_t(m_1)=X_t(m_2)=x$. Since $Y^k\in\COPT$ and $m_1,m_2\not\in
N^k$, we get
$$
  (m_2-m_1) \big( Y^k(m_2)-Y^k(m_1) \big)\GS 0,
$$
which implies a similar inequality for $W^k$. Passing to the limit
$k\rightarrow \infty$, we obtain the following statement: for all
$W\in\TAN_{X_t}\COPT$ and for all $x\in\mathscr{O}$, we have
$$
  (m_2-m_1) \big( W(m_2)-W(m_1) \big)\GS 0
$$
for $\mu$-a.e.\ $m_1,m_2\in L_x$. For $W\in-\TAN_{X_t}\COPT$ we have
the opposite inequality. 
\medskip

{\bf Step~3.} As shown above, the velocity $V_t \in \TAN_{X_t}\COPT
\cap (-\TAN_{X_t}\COPT)$. Then there exists a function $v_t\colon
\R\longrightarrow\R$ with $V_t = v_t\circ X_t$ $\mu$-a.e. Indeed for
all $x\in\R\setminus\mathscr{O}$ there exists at most one $m\in L_x$
such that $X_t(m)=x$, so we define
$$
  v_t(x) := \begin{cases} 
      V_t(m) & \text{if such an $m$ exists, and}
\\
      0 & \text{otherwise.}
    \end{cases}
$$
For all $x\in\mathscr{O}$ we can pick a generic $m\in L_x$ and define
$v_t(x) := V_t(m)$ because then $V_t$ is constant $\mu$-a.e.\ in
$L_x$, as shown in Step~2. By construction, we have
$$
  V_t(m) = v_t(X_t(m))
  \quad\text{for $\mu$-a.e.\ $m\in\R\setminus N_{X_t}$.}
$$
Using that $\mu(N_{X_t})=0$, we get the claim. Finally, since
$\|V_t\|_{\L^2(\R, \mu)} = \|v_t\|_{\L^2(\R, \RHO_t)}$ with $\RHO_t :=
X_t\#\mu$, we have that $v_t\in\L^2(\R,\RHO_t)$. Then we use
Lemma~\ref{L:CHR}.
\end{proof}

\begin{proof}[Proof of Theorem~\ref{T:PGD}]
Notice first that $\H_{X_t}$ is a closed subspace of $\L^2(\R,\mu)$,
and that the velocity $V_t$ is an element of that subspace for a.e.\
$t\in\R$; see Lemma~\ref{L:CONTAIN}. Combining Lemma~\ref{L:SUBSET}
with Theorem~\ref{T:HARAUX}, we obtain that $\langle \bar{V}-V_t, W
\rangle \LS 0$ for all $W\in\H_{X_t}$, which implies that $V_t$ is the
orthogonal projection of $\bar{V}$ onto $\H_{X_t}$.

There exists $v_t \in \L^2(\R,\RHO_t)$ with $\RHO_t := X_t\#\mu$, such
that $V_t = v_t\circ X_t$ $\mu$-a.e.; see the proof of
Lemma~\ref{L:CONTAIN}. Notice that since $\bar{X} \in \COPT$, by
assumption, we have $X_0 = \bar{X}$ and therefore $\RHO_0 = X_0\#\mu =
\bar{X}\#\mu = \bar{\RHO}$. For any $\varphi \in
\D([0,\infty)\times\R)$, we get
\begin{align*}
  -\int_\R \varphi(0,x) \,\bar{\RHO}(dx)
    & = - \int_\R \varphi(0,\bar{X}(m)) \,\mu(dm)    
\\
    & = \int_0^\infty \frac{d}{dt} \bigg( \int_\R \varphi(t,X_t(m)) \,\mu(dm) 
      \bigg) \,dt.
\end{align*}
Recall that the map $t\mapsto X_t$ is differentiable for a.e.\
$t\in\R$, with velocity $V_t = v_t\circ X_t$ where $v_t \in
\L^2(\R,\RHO_t)$. For a.e.\ $t\in\R$ we can therefore write
\begin{align*}
  & \frac{d}{dt} \bigg( \int_\R \varphi(t,X_t(m)) \,\mu(dm) \bigg)
\\
  & \qquad
    = \int_\R \Big( \partial_t\varphi(t,X_t(m))
      + V_t(m) \partial_x\varphi(t,X_t(m)) \Big) \,\mu(dm)
\\
  & \qquad
    = \int_\R \Big( \partial_t\varphi(t,X_t(m))
      + v_t(X_t(m)) \partial_x\varphi(t,X_t(m)) \Big) \,\mu(dm)
\\
  & \qquad
    = \int_\R \Big( \partial_x \varphi(t,x)
      + v_t(x) \partial_x\varphi(t,x) \Big) \,\RHO_t(dx).
\end{align*}
It follows that for all $\varphi \in \D([0,\infty)\times\R)$, we have
\begin{align*}
  -\int_\R \varphi(0,x) \,\bar{\RHO}(dx)
    & = \int_0^ \infty \int_\R \Big( \partial_x \varphi(t,x)
      + v_t(x) \partial_x\varphi(t,x) \Big) \,\RHO_t(dx) \,dt.
\end{align*}
This proves that $(\RHO_t,v_t)$ satisfies the continuity equation in
the distributional sense.

Similarly, for any test function $\zeta \in \D([0,\infty)\times\R)$,
we can write
\begin{align*}
  -\int_\R \zeta(0,x) \bar{v}(x) \,\bar{\RHO}(dx)
    & = -\int_\R \zeta(0,\bar{X}(m)) \bar{v}(\bar{X}(m)) \,\mu(dm)
\\
    & = -\int_\R \zeta(0,\bar{X}(m)) \bar{V}(m) \,\mu(dm)
\\
    & = \int_0^\infty \frac{d}{dt} \bigg( 
      \int_\R \zeta(t, X_t(m)) \bar{V}(m) \,\mu(dm) \bigg) \,dt.
\end{align*}
We used the fact that, by assumption, the initial velocity is of the
form $\bar{V} = \bar{v}\circ\bar{X}$ (thus $\bar{V} \in \H_{\bar{X}}$;
see Lemma~\ref{L:CHR}). For a.e.\ $t\in\R$, we can write
\begin{align*}
  & \frac{d}{dt} \bigg( \int_\R \zeta(t, X_t(m)) \bar{V}(m) 
    \,\mu(dm) \bigg)
\\
  & \qquad
    = \int_\R \Big( \partial_t\zeta(t,X_t(m))
      + V_t(m) \partial_x\zeta(t,X_t(m)) \Big) \bar{V}(m)\,\mu(dm)
\\
  & \qquad
    = \int_\R \Big( \partial_t\zeta(t,X_t(m))
      + v_t(X_t(m)) \partial_x\zeta(t,X_t(m)) \Big) \bar{V}(m)\,\mu(dm).
\end{align*}
Since $V_t$ is the orthogonal projection of $\bar{V}$ onto $\H_{X_t}$,
we get that
\begin{align*}
  & \int_\R \Big( \partial_t\zeta(t,X_t(m))
      + v_t(X_t(m)) \partial_x\zeta(t,X_t(m)) \Big) \bar{V}(m)\,\mu(dm)
\\
  & \qquad
    = \int_\R \Big( \partial_t\zeta(t,X_t(m))
      + v_t(X_t(m)) \partial_x\zeta(t,X_t(m)) \Big) V_t(m) \,\mu(dm)
\\
  & \qquad
    = \int_\R \Big( \partial_t\zeta(t,X_t(m))
      + v_t(X_t(m)) \partial_x\zeta(t,X_t(m)) \Big) v_t(X_t(m)) \,\mu(dm)
\\
  & \qquad
    = \int_\R \Big( \partial_t\zeta(t,x) 
      + v_t(x) \partial_x\zeta(t,x) \Big) v_t(x) \,\RHO_t(dx).
\end{align*}
using the fact that $\partial_t\zeta(t,X_t(m)) + v_t(X_t(m))
\partial_x\zeta(t,X_t(m))$ is an element of $\H_{X_t}$. It follows
that for all $\zeta \in \D([0,\infty)\times\R)$, we have
\begin{align*}
  -\int_\R \zeta(0,x) \bar{v}(x) \,\bar{\RHO}(dx)
    & = \int_0^ \infty \int_\R \Big( \partial_t \zeta(t,x)
      + v_t(x) \partial_x\zeta(t,x) \Big) v_t(x) \,\RHO_t(dx) \,dt.
\end{align*}
This proves that $(\RHO_t,v_t)$ satisfies the momentum equation
distributionally.
\end{proof}

\begin{remark}
Theorem~\ref{T:PGD} covers the special case $\mu=\bar{\RHO}$, for
which $(\bar{X}, \bar{V})=(\ID, \bar{v})$. More generally, if there
are two sets of initial data $(\bar{\RHO}_i,\bar{v}_i)$, $i=1..2$, and
if $(\bar{X}_i, \bar{V}_i)$ are the monotone transport maps and
initial velocities corresponding to the reference measure
$\mu\in\SP_2(\R)$, then the transport maps $X_{i,t}$ defined as in
\eqref{E:PRF} satisfy
$$
  \|X_{1,t}-X_{2,t}\|_{\L^2(\R,\mu)} 
    \LS \|\bar{X}_1-\bar{X}_2\|_{\L^2(\R,\mu)} 
      + |t| \|\bar{V}_1-\bar{V}_2\|_{\L^2(\R,\mu)}
$$
for all $t\in\R$ since the metric projection onto closed convex
subsets of Hilbert spaces is a contraction. In particular, this
implies the uniquness of the transport map $X_t$, from which the
uniqueness of the induced density $\RHO_t := X_t\#\mu$ follows. The
Eulerian velocities $v_t$ are determined by the orthogonal projection
of $\bar{V}$ onto the space $\H_{X_t}$, which is also unique. We
conclude that within the framework of solutions obtained from
\eqref{E:PRF}, we obtain both existence of solutions to
\eqref{E:PGDE}, and uniqueness.
\end{remark}


\section{Force field}\label{S:FF}

The analysis above can also be applied to the case when the momentum
equation in \eqref{E:PGDE} is augmented with a force generated by the
fluid itself: We consider
\begin{equation}\label{E:PGDEF}
  \left.\begin{aligned}
    \partial_t\RHO + \partial_x(\RHO v) &= 0
\\
    \partial_t(\RHO v) + \partial_x(\RHO v^2) &= f[\RHO]
  \end{aligned}\right\}
  \quad\text{in $[0,\infty)\times\R$,}
\end{equation}
where $f \colon \SP_2(\R) \longrightarrow \SM(\R)$ is the force field,
with $\SM(\R)$ the space of finite, signed Borel measures. This
equation has been studied in \cite{BrenierGangboSavareWestdickenberg}
for a suitable class of forces. Instead of stating our result in the
general terms of \cite{BrenierGangboSavareWestdickenberg} (which is
possible), we focus here on the special case of the Euler-Poisson
system, for simplicity. Thus
\begin{equation}\label{E:EPCASE}
  f[\RHO] = -\RHO \, \partial_{x}  q_{\RHO},
  \quad\text{with $q_\RHO$ solution of $-\partial^{2}_{xx} q_{\RHO} = \RHO$.}
\end{equation}
As a first step, one needs to identify a {\em Lagrangian} force field
representing $f$. That means, we are looking for a functional $F
\colon \COPT \longrightarrow \L^2(\R,\mu)$ such that
$$
  \int_\R \varphi(x) \,f[\RHO](dx)
    = \int_\R \varphi(X(m)) F[X](m) \,\mu(dm)
  \quad\text{for all $\varphi \in \CB(\R)$,}
$$
where $\RHO \in \SP_{2}(\R)$ and $X \in \COPT$ are related by the
identity $X\# \mu = \RHO$. Whenever $f[\RHO]$ is absolute continuous
with respect to $\RHO$, and so can be expressed in terms of the
Radon-Nikodym density $f_\RHO \in \L^2(\R,\RHO)$ as $f[\RHO] = f_\RHO
\,\RHO$, one can choose
$$
  F[X] := f_\RHO \circ X
  \quad\text{for all $X \in \COPT$ with $X\#\mu = \RHO$.}   
$$
To every $X \in \C([0,\infty), \COPT)$ (with topology induced by the
$\L^{2}(\R,\mu)$-norm) we can associate a new variable playing the
role of a modified velocity:
\begin{equation}\label{E:YT}
  Y_t : = \bar{V} + \int_0^t F[X_s] \,ds
  \quad\text{for all $t \in [0,\infty)$.}
\end{equation}
The evolution $X_t$ for the system \eqref{E:PGDEF} is then
characterized by the identity
\begin{equation}\label{E:SOLUTION}
  X_t = \PROJ{\COPT}\bigg( \bar{X} + \int_0^t Y_s \,ds \bigg)
  \quad\text{for all $t \in [0,\infty)$.}
\end{equation}
The existence of a curve $X \in \LIP([0,\infty),\COPT)$ satisfying
\eqref{E:YT}/\eqref{E:SOLUTION} is especially easy to establish in the
Euler-Poisson case \eqref{E:EPCASE} (in the general case one can use a
fixed point argument as in \cite{BrenierGangboSavareWestdickenberg}):
We assume for simplicity that $\mu = \LEB^1|_{[0,1]}$. Then
$$
  f_\RHO(x) = -\frac{1}{2} \Big( m^-_\RHO(x) + m^+_\RHO(x) - 1 \Big)
  \quad\text{for all $x\in\R$}
$$
(see Example~6.10 in \cite{BrenierGangboSavareWestdickenberg}), where
$m^\pm_\RHO$ are cumulative distribution functions:
$$
  m^-_\RHO(x) := \RHO\big( (-\infty,x) \big),
  \quad
  m^+_\RHO(x) := \RHO\big( (-\infty,x] \big)
$$
for $x\in\R$. By construction, it then follows that
\begin{equation}\label{E:FORCEEP}
  F[X](m) = -\frac{1}{2} (2m - 1)
  \quad\text{for all $m \in [0,1]$.}
\end{equation}
In particular, the functional $F[X]$ does even not depend on
$X\in\COPT$ anymore, and so \eqref{E:YT} can be computed explicitly.
Global existence of $X$ satisfying \eqref{E:SOLUTION} follows, and
Lipschitz continuity in time is a consequence of the fact that the
metric projection $\PROJ{\COPT}$ is a contraction in $\L^2(\R,\mu)$.
To prove that \eqref{E:SOLUTION} generates a solution of
\eqref{E:PGDEF}, we need the following result, which remains true for
general $Y \in \C([0,\infty), \L^{2}(\R,\mu))$ defining $X$ through
\eqref{E:SOLUTION}, even those not necessarily given in terms of $X$.

\begin{proposition}\label{P:FORCE}
Consider $\bar X \in \COPT$ and $Y \in \C([0,\infty), \L^{2}(\R,\mu))$
and define
$$
  X_t : = \PROJ{\COPT}\bigg( \bar{X} + \int_0^t Y_s \,ds \bigg)
  \quad\text{for all $t \in [0,\infty)$.}
$$
Then the derivative of $X_t$, denoted by $V_t$, exists and satisfies
\begin{equation}\label{E:SPEED}
  V_t = \PROJ{\H_{X_t}}(Y_t)
  \quad\text{for a.e.\ $t \in [0,\infty)$.}
\end{equation}
\end{proposition}

\begin{proof}
Each step of the proof follows the same ideas of the corresponding
result for the pressureless gas dynamics equations. We proceed in two
steps. 
\medskip

{\bf Step~1.} Repeating word by word the proof of Theorem
\ref{T:HARAUX}, replacing $(t+h_{n})\bar V$ by $\int_0^{t+h_n} Y(s)ds$
for all $t,h_n \in [0,\infty)$, one obtains the analogous results:
\begin{align*}
  & V_t \in \TAN_{X_t} \COPT \cap \bigg[ \bigg( \bar{X} 
    + \DST\int_{0}^{t}Y_s \,ds \bigg)-X_t \bigg]^\perp,
\\
  & \vphantom{\bigg]^\perp}
    \langle Y_t - V_t, V_t \rangle \GS 0,
\\
  & \langle Y_t - V_t, W \rangle \LS 0
  \quad\text{for all $W \in T_{X_t}\COPT \cap \bigg[ \bigg( \bar{X} 
    + \DST\int_{0}^{t}Y_s \,ds \bigg)-X_t \bigg]^\perp$.}
\end{align*}
\medskip

{\bf Step~2.} For any $t \in [0,\infty)$ for which $X_t$ is strongly
differentiable (these are a.e.\ $t$) we then deduce from Lemma
\ref{L:CONTAIN} that the velocity satisfies
$$
  V_t \in \TAN_{X_t}\COPT \cap (-\TAN_{X_t}\COPT) \subset \H_{X_t}. 
$$
Repeating the proof of Lemma~\ref{L:SUBSET}, replacing $t\bar V$ by
$\int_0^t Y_s \,ds$, one gets $\H_{X_t} \subset \mathbb{S}_{X_t}
\COPT$, where the definition of $\mathbb{S}_{X_t}\COPT$ changes
slightly from the previous one:
$$
  \mathbb{S}_{X_t}\COPT := \text{$\L^2(\R,\mu)$-closure of $T_{X_t}\COPT \cap 
    \bigg[ \bigg( \bar{X}+\int_0^t Y_s \,ds \bigg)-X_t \bigg]^\perp$.}
$$
We now deduce from Step~1. that 
$$
  \langle Y_t-V_t,V_t \rangle = 0,
  \quad
  \langle Y_t-V_t,W \rangle \LS 0
  \quad\text{for all $W \in \H_{X_t}$.}
$$
Since $V_t \in \H_{X_t}$, we obtain that $V_t$ is the orthogonal
projection of $Y_t$ onto $\H_{X_t}$.
\end{proof}

Applying Proposition~\ref{P:FORCE}, we can prove the following
existence result.

\begin{theorem}[Global Existence Euler-Poisson]\label{T:PGDFF}
Let initial data $\bar{\RHO}\in\SP_2(\R)$ and $\bar{v} \in \L^2(\R,
\bar{\RHO})$ be given. For suitable reference measure $\mu \in
\SP_2(\R)$, let $\bar{X} \in \COPT$ be the unique monotone transport
map with $\bar{X}\#\mu = \bar{\RHO}$. Define $\bar{V} := \bar{v} \circ
\bar{X}$, and let
\begin{equation}\label{E:PRFF}
   X_t := \PROJ{\COPT}\bigg( \bar{X} + \int_0^t Y_s \,ds \bigg),
  \quad
  Y_t = \bar{V} + \int_0^t F[X_s] \,ds
\end{equation}
for any $t\in\R$, where $F$ is given by \eqref{E:FORCEEP}. Then $X_t$
is strongly differentiable for a.e.\ $t \in [0,\infty)$ and $V_t :=
\dot{X}_t$ can be written in the following form: there exists a
velocity $v_t \in \L^2(\R,\RHO_t)$ with $\RHO_t := X_t\#\mu$, such
that $V_t = v_t \circ X_t$. The pair $(\RHO_t, v_t)$ is a weak
solution of the conservation law \eqref{E:PGDEF} with force
\eqref{E:EPCASE}.
\end{theorem}

The proof is omitted; see the proof of Theorem 3.5 in
\cite{BrenierGangboSavareWestdickenberg} for details.


\end{document}